\documentclass{amsart}

\usepackage{amsmath,float,graphicx}

\newtheorem{theorem}{Theorem}[section]
\newtheorem{lemma}[theorem]{Lemma}

\newtheorem{proposition}[theorem]{Proposition}
\theoremstyle{remark}

\title{The cap set problem: Up to dimension 7}

\author{Henry (Maya) Robert Thackeray}

\address{Department of Mathematics and Applied Mathematics, University of Pretoria, Pretoria, 0002 South Africa, maya.thackeray@up.ac.za, mayart314@outlook.com}

\begin{document}
\begin{abstract}
An \emph{$s$-cap $n$-flat} is given by a set of $s$ points, no three of which are on a common line, in an $n$-dimensional affine space over the field of three elements. The cap set problem in dimension $n$ is: what is the maximum $s$ such that there is an $s$-cap $n$-flat?

The first two papers in this series of articles considered the cap set problem in dimensions up to and including $5$. In this paper, which is the third in the series, we consider dimensions $6$ and $7$: we prove that every $110$-cap $6$-flat is a $112$-cap $6$-flat minus two cap points, and that there are no $289$-cap $7$-flats.
\end{abstract}
\maketitle



\section{Introduction}

This paper is the third in a series that investigates the cap set problem. The cap set problem in dimension $n$ is: what is the maximum $s$ such that there is an \emph{$s$-cap $n$-flat} (also called an \emph{$n$-dimensional cap} of \emph{size} $s$), that is, a pair $(S,F)$ such that $F$ is an $n$-dimensional affine space over $\mathbb{Z}/3\mathbb{Z}$ and the size-$s$ subset $S$ of $F$ contains no three points on a common line? For dimensions up to and including $6$, the cap set problem is solved; see Davis and Maclagan \cite{DM03}, Edel et al.\@ \cite{EFLS02}, Ellenberg and Gijswijt \cite{EG17}, and Potechin \cite{P08} for background reading and previously known results for the problem.

The first paper of this series (Thackeray \cite{T21}) classified certain caps in dimensions up to and including $4$, and the second paper (Thackeray \cite{T22}) classified caps of size at least $41$ in dimension $5$; we take those two papers as known. The current paper investigates dimensions $6$ and $7$: we prove that every $110$-cap $6$-flat is a $112$-cap $6$-flat minus two cap points, and that there are no $289$-cap $7$-flats.

\section{Lemmas}

We start by recalling known results and proving useful lemmas.

\begin{theorem}[Large 5-dimensional caps]\label{n5s42to45}
Let $C$ be a $5$-dimensional cap. If we have $|C| \geq 43$, then $C$ is a $45$-cap $5$-flat minus at most two cap points; if $|C| = 42$, then $C$ is a $\triangle{}686$ or $C$ is a $45$-cap $5$-flat minus three cap points. In a $\triangle{}686$, the hyperplane point counts are as in Table \ref{tbl686dirs}.
\end{theorem}
\begin{table}
\[\begin{tabular}{cccc}
\hline
\textrm{Point count} & \textrm{Hyperplane directions} & \textrm{Point count} & \textrm{Hyperplane directions}\\
\hline
\{20,16,6\} & 3 & \{16,15,11\} & 24\\
\{18,18,6\} & 4 & \{16,14,12\} & 36\\
\{18,17,7\} & 18 & \{15,15,12\} & 3\\
\{18,12,12\} & 6 & \{14,14,14\} & 27\\
\hline
\end{tabular}\]
\caption{Hyperplane point counts of a $\Delta{}686$.\label{tbl686dirs}}
\end{table}
\begin{proof}
See Thackeray \cite{T22}; the last sentence was checked directly using a representative $\triangle{}686$.
\end{proof}

\begin{lemma}[Hyperplane directions of 45-cap 5-flat]\label{n5s45hypclassify}
Let $C$ be a $45$-cap $5$-flat. Of all the $3$-flat directions $D$ of $C$, exactly $45$ are such that the nine $3$-dimensional caps in the $3$-flats of $D$ are eight square pyramids and one tetrahedron plus centre. There is a unique line direction $L$ of $C$ such that in each of those $45$ directions $D$, each of the nine $3$-flats is the union of some lines in $L$. The cap $C$ has
\begin{itemize}
\item[(i)] exactly $10$ $\{18,9,18\}$ hyperplane directions in which the $18$-cap $4$-flats are both $882A_{1}$ caps,
\item[(ii)] exactly $45$ $\{18,9,18\}$ hyperplane directions in which the $18$-cap $4$-flats are both $882A_{2}$ caps,
\item[(iii)] exactly $30$ $\{15,15,15\}$ hyperplane directions in which each $4$-flat is a union of lines in $L$, and
\item[(iv)] exactly $36$ $\{15,15,15\}$ hyperplane directions in which each $4$-flat is not a union of lines in $L$,
\end{itemize}
and for each two hyperplane directions in the same one of those four categories, some symmetry of $C$ sends one of the two hyperplane directions to the other.
\end{lemma}
\begin{proof}
This was checked directly using a representative $45$-cap $5$-flat.
\end{proof}

\begin{lemma}[Replacing cap points in an 18-cap 4-flat]\label{n4s18882A2repl3pts}
Let $(S,F)$ be an $882A_{2}$. Let $n \in \{1,2,3\}$. If $A$ and $B$ are subsets of $S$ and $F - S$ respectively such that $|A| = |B| = n$ and $((S - A) \cup B,F)$ is a cap, then that cap contains a $9$-cap $3$-flat (and therefore is neither an $882A_{1}$ nor an $882A_{2}$).
\end{lemma}
\begin{proof}
Without loss of generality, let $(S,F)$ be the $882A_{2}$ in Thackeray \cite[Figure 42, ``Another $882A_{2}$'']{T21}. For each $n \in \{1,2,3\}$, and for each subset $S_{-}$ of $S$ such that $|S_{-}| = n$, a computer search found each possible subset $S_{+}$ of $F - (S - S_{-})$ such that $|S_{+}| = n$ and $((S - S_{-}) \cup S_{+},F)$ is a cap with no $9$-cap $3$-flat (which corresponds to $A = S_{-} - S_{+}$ and $B = S_{+} - S_{-}$). In the cases $n = 1$, $n = 2$, and $n = 3$, there are respectively $\binom{18}{1}$, $\binom{18}{2}$, and $\binom{18}{3}$ solutions such that $S_{+} = S_{-}$ (so $A = B = \emptyset$); the computer search confirmed that there are no more solutions for each of those values of $n$. (A further computer search confirmed that for $n = 4$, additional solutions do exist.)
\end{proof}

Each $882A_{2}$ has the following properties.
\begin{itemize}
\item[(a)] The $882A_{2}$ has exactly one $2$-flat direction in which each $2$-flat has exactly two cap points. This is the \emph{nine-$2$s} $2$-flat direction.
\item[(b)] The $882A_{2}$ has exactly one $\{8,5,5\}$ hyperplane direction in which the $8$-cap $3$-flat is a cube. This direction is the \emph{$855$ cube direction}.
\item[(c)] The $882A_{2}$ has exactly one $\{8,8,2\}$ hyperplane direction in which both $8$-cap $3$-flats are cubes. This direction is the \emph{$882$ cube direction}.
\item[(d)] There is some independent pair $(x_{1},x_{2})$ of co-ordinates (determined up to negating $x_{1}$ and/or $x_{2}$) such that the $x_{1}$-hyperplane (respectively, $x_{2}$-hyperplane) direction is the $855$ cube direction (respectively, the $882$ cube direction) and the point count of the $882A_{2}$ for $(x_{1},x_{2})$ is
\[\left(\begin{array}{ccc}
2 & 4 & 2\\
1 & 0 & 1\\
2 & 4 & 2
\end{array}\right).\]
For such a pair $(x_{1},x_{2})$, the $4$-cap $2$-flats $(x_{1},x_{2}) = (0,\pm 1)$ are translations of each other; they are the \emph{standard squares} of the $882A_{2}$. The midpoints of the cap line segments in the four $2$-flats $(x_{1},x_{2}) = (\pm 1,\pm 1)$ form a square -- which is the \emph{square of midpoints} of the $882A_{2}$ -- and the $2$-flat containing that square also contains the centres of the standard squares and the cap points in the $2$-flats $(x_{1},x_{2}) = (\pm 1,0)$.
\end{itemize}

\begin{lemma}[Hyperplanes that are 18-cap 4-flats]\label{n5882855para}
In each $45$-cap $5$-flat, for each $\{18,18,9\}$ hyperplane direction $D$ in which the $18$-cap $4$-flats are $882A_{2}$ caps, we have the following.
\begin{itemize}
\item[(a)] The nine-$2$s $2$-flat directions of the $18$-cap $4$-flats in $D$ are parallel.
\item[(b)] The $882$ cube direction of each $18$-cap $4$-flat in $D$ is parallel to the $855$ cube direction of the other $18$-cap $4$-flat in $D$.
\item[(c)] The side directions of each standard square of each $18$-cap $4$-flat in $D$ are parallel to the diagonal directions of each standard square of the other $18$-cap $4$-flat in $D$.
\end{itemize}
\end{lemma}
\begin{proof}
This follows from Thackeray \cite[Lemma 3.1]{T22}, which indicates that a $45$-cap $5$-flat with a chosen $\{18,18,9\}$ hyperplane direction $D$ in which the $18$-cap $4$-flats are $882A_{2}$ caps is unique up to isomorphisms under which the image of each hyperplane in $D$ is a hyperplane in $D$.
\end{proof}

\begin{lemma}[Two 18-cap 4-flats minus cap points]\label{n5Double882A2Min3}
Consider a $5$-dimensional cap $C$ with co-ordinate $x_{1}$. Let the cap $\widetilde{C}$, with the same underlying $5$-flat as $C$, be the union of two $882A_{2}$ caps in the $4$-flats $x_{1} = \pm 1$ respectively. Suppose that the union of the two $4$-dimensional subcaps $x_{1} = \pm 1$ of $C$ is obtained from that of $\widetilde{C}$ by removing $n$ cap points, for some nonnegative integer $n$.
\begin{itemize}
\item[(a)] Suppose that the $882A_{2}$ caps $x_{1} = \pm 1$ in $\widetilde{C}$ are translations of each other. If $n$ is $0$, $1$, $2$, or $3$, then the number of cap points in the $4$-flat $x_{1} = 0$ of $C$ is $0$, at most $1$, at most $2$, or at most $3$ respectively.
\item[(b)] Suppose that the $882A_{2}$ caps $x_{1} = \pm 1$ in $\widetilde{C}$ are point reflections of each other. If $n$ is $0$, $1$, $2$, or $3$, then the number of cap points in the $4$-flat $x_{1} = 0$ of $C$ is at most $4$, at most $4$, at most $5$, or at most $5$ respectively.
\end{itemize}
Each of those upper bounds cannot be lowered.
\end{lemma}
\begin{proof}
Consider $\widetilde{C}$. For each point $Q$ in the $4$-flat $x_{1} = 0$, find the number $n(Q)$ of cap line segments $L$ with one cap point in each of the $4$-flats $x_{1} = \pm 1$ such that $Q$ is the midpoint of $L$, and for each $a \in \{\pm 1\}$, let $S_{a}(Q)$ be the set of cap points $P$ in the $882A_{2}$ cap $x_{1} = a$ such that the midpoint of $PQ$ is a cap point in the $882A_{2}$ cap $x_{1} = -a$.

For part (a), it was checked that if the $882A_{2}$ cap $x_{1} = 1$ is the image of the $882A_{2}$ cap $x_{1} = -1$ under the translation map $U$, then the following statements hold:
\begin{itemize}
\item[(i)] there are exactly $18$ points $Q$ with $n(Q) = 1$, namely, the images under $U$ of the cap points in $x_{1} = 1$; for these points $Q$, the cap line segment with midpoint $Q$ is $U(Q)U^{-1}(Q)$;
\item[(ii)] there are exactly four points $Q$ with $n(Q) = 2$, which form the image under $U$ of the square of midpoints of the $882A_{2}$ cap $x_{1} = 1$; these four points $Q$ are the midpoints of eight disjoint cap line segments $L$, and for each of these four points $Q$ and each $a \in \{\pm 1\}$, the cap line segment $S_{a}(Q)$ is not contained in any $2$-flat in the nine-$2$s $2$-flat direction of the $882A_{2}$ cap $x_{1} = a$, and $Q$ is the midpoint of the image of $S_{a}(Q)$ under $U^{a}$; and
\item[(iii)] all other points $Q$ satisfy $n(Q) \geq 4$.
\end{itemize}
The result in part (a) follows. (After at most three cap points are removed from $\widetilde{C}$ to obtain the $4$-flats $x_{1} = \pm 1$ of $C$, consider which points $Q$, and how many points $Q$, can be cap points in the $4$-flat $x_{1} = 0$ of $C$.)

For part (b), it was checked that if the square of midpoints of the $882A_{2}$ cap $x_{1} = 1$ is the image of the square of midpoints of the $882A_{2}$ cap $x_{1} = -1$ under the translation map $U$, then the following statements hold:
\begin{itemize}
\item[(i)] there are exactly four points $Q$ with $n(Q) = 0$, namely, the images under $U$ of the cap points in the square of midpoints of $x_{1} = 1$;
\item[(ii)] there are exactly four points $Q$ with $n(Q) = 2$, they are the midpoints of eight disjoint cap line segments $L$, those points $Q$ form a square $S$ that is a translation of each of the standard squares of the $882A_{2}$ caps $x_{1} = \pm 1$, and the centre of $S$ is the image under $U$ of the centre of the square of midpoints of $x_{1} = 1$;
\item[(iii)] there are exactly $24$ points $Q$ with $n(Q) = 3$; for each $Q_{2}$ and $Q_{3}$ with $n(Q_{2}) = 2$ and $n(Q_{3}) = 3$, the intersection of $S_{1}(Q_{2}) \cup S_{-1}(Q_{2})$ and $S_{1}(Q_{3}) \cup S_{-1}(Q_{3})$ has at most one point or is a cap line segment with midpoint $Q_{3}$; for each two different points $Q_{3}$ and $Q_{3}'$ with $n(Q_{3}) = n(Q_{3}') = 3$, if the intersection of $S_{1}(Q_{3}) \cup S_{-1}(Q_{3})$ and $S_{1}(Q_{3}') \cup S_{-1}(Q_{3}')$ has more than two points, then
\begin{itemize}
\item[(1)] that intersection has exactly three points,
\item[(2)] that intersection determines $\{Q_{3},Q_{3}'\}$, and
\item[(3)] the midpoint $Q_{0}$ of the line segment $Q_{3}Q_{3}'$ satisfies $n(Q_{0}) = 0$;
\end{itemize}
and
\item[(iv)] all other points $Q$ satisfy $n(Q) \geq 4$.
\end{itemize}
The result in part (b) follows as in part (a).
\end{proof}

\section{Dimension 6}

Throughout this section, $C$ is a $6$-dimensional cap with some co-ordinate $\widetilde{x}_{1}$.

Potechin \cite{P08} proved that there is a unique $112$-cap $6$-flat up to isomorphism. A representative $112$-cap $6$-flat is shown in Figure \ref{fign6s112rep}.

\begin{figure}
\includegraphics{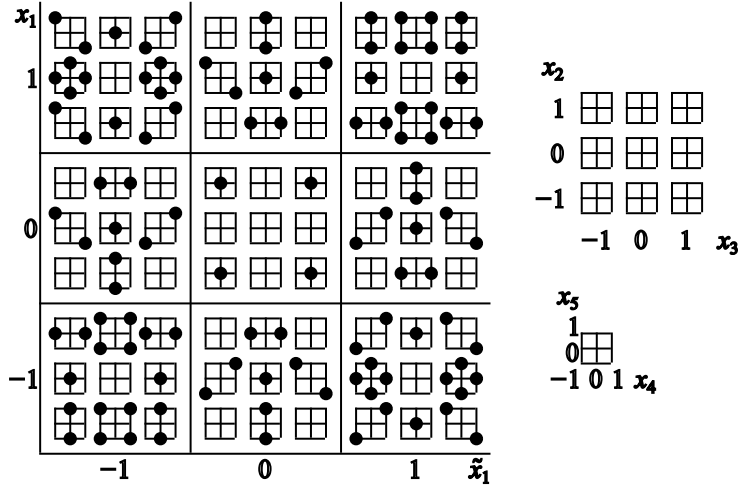}
\caption{A $112$-cap $6$-flat. Each of the four $18$-cap $4$-flats $(\widetilde{x}_{1},x_{1}) = (\pm 1,\pm 1)$ is an $882A_{2}$.\label{fign6s112rep}}
\end{figure}

\begin{lemma}[The 112-cap 6-flat]\label{n6s112str}
Suppose that $C$ is a $112$-cap $6$-flat.
\begin{itemize}
\item[(a)] The numbers of $\{45,45,22\}$ hyperplane directions and $\{40,36,36\}$ hyperplane directions of $C$ are $56$ and $308$ respectively.
\item[(b)] For every $\{40,36,36\}$ hyperplane direction $D$ of $C$, there are co-ordinates $y_{1}$ and $y_{2}$ of $C$ with $(y_{1},y_{2})$ independent such that each of the $y_{1}$- and $y_{2}$-hyperplane directions of $C$ has point count $\{45,45,22\}$ and the $(y_{1} - y_{2})$-hyperplane direction of $C$ is $D$.
\item[(c)] In each $4$-flat direction of $C$, there are (i) four $18$-cap $4$-flats, four $9$-cap $4$-flats, and one $4$-cap $4$-flat; (ii) six $15$-cap $4$-flats, one $10$-cap $4$-flat, and two $6$-cap $4$-flats; or (iii) one $16$-cap $4$-flat and eight $12$-cap $4$-flats. The numbers of $4$-flat directions in cases (i), (ii), and (iii) are respectively $\binom{56}{2} = \mbox{$1$ $540$}$, \mbox{$3$ $696$}, and \mbox{$5$ $775$}.
\item[(d)] In part (c), the $4$-flat directions in case (i) are precisely the $4$-flat directions $D$ that are obtained from the unordered pairs $\{D_{1},D_{2}\}$ of different $\{45,45,22\}$ hyperplane directions of $C$ by letting each $4$-flat in $D$ be the intersection of a hyperplane in $D_{1}$ with a hyperplane in $D_{2}$.
\end{itemize}
\end{lemma}
\begin{proof}
Parts (a) to (c) were checked directly using a representative $112$-cap $6$-flat. From parts (a) and (c), we deduce that each $4$-flat direction $D$ has co-ordinates $y_{1}$ and $y_{2}$ such that the $4$-flat point count of $C$ for $(y_{1},y_{2})$ is
\[\left(\begin{array}{ccc}
18 & 9 & 18\\
9 & 4 & 9\\
18 & 9 & 18
\end{array}\right), \left(\begin{array}{ccc}
15 & 6 & 15\\
15 & 10 & 15\\
15 & 6 & 15
\end{array}\right), \textrm{ or } \left(\begin{array}{ccc}
12 & 12 & 12\\
12 & 16 & 12\\
12 & 12 & 12
\end{array}\right).\]
Part (d) follows.
\end{proof}

\begin{lemma}[Refining \{45,*,45\}]\label{n6180918882A2}
Suppose that the two $5$-flats $\widetilde{x}_{1} = \pm 1$ have exactly $45$ cap points each. It follows that for some co-ordinate $x_{1}$ of $C$, the point count of $C$ for $(\widetilde{x}_{1},x_{1})$ is
\[\left(\begin{array}{ccc}
18 & * & 18\\
9 & * & 9\\
18 & * & 18
\end{array}\right)\]
and the four $18$-cap $4$-flats $(\widetilde{x}_{1},x_{1}) = (\pm 1,\pm 1)$ are all $882A_{2}$ caps.
\end{lemma}

\begin{proof}
Each $45$-cap $5$-flat $\widetilde{x}_{1} = \pm 1$ has exactly $90$ dual vectors that correspond to $\{18,9,18\}$ hyperplane directions in which the $18$-cap $4$-flats are two $882A_{2}$ caps (two dual vectors for each of $45$ such hyperplane directions). Translate (the arguments of) such dual vectors of $\widetilde{x}_{1} = -1$ (respectively, $\widetilde{x}_{1} = 1$) to obtain a set $S_{-1}$ (respectively, $S_{1}$) of $90$ dual vectors of the $5$-flat $\widetilde{x}_{1} = 0$. Write $F$ for the space of all $3^{5}$ dual vectors of the $5$-flat $\widetilde{x}_{1} = 0$.

For each $a \in \{\pm 1\}$, among the $121$ hyperplane directions of $(S_{a},F)$, the numbers of $\{45,45,0\}$, $\{36,36,18\}$, $\{30,30,30\}$, and $\{27,27,36\}$ hyperplane directions are respectively $1$, $10$, $90$, and $20$. The $\{45,45,0\}$ hyperplane direction $D$ of $(S_{a},F)$ corresponds via duality to the axis line direction $L$ of the $45$-cap $5$-flat $\widetilde{x}_{1} = a$: two different dual vectors $f_{1}$ and $f_{2}$ in $F$ are in the same hyperplane in $D$ if and only if, in the hyperplane direction corresponding to the dual vector $f_{1} - f_{2}$, each hyperplane is a union of lines in $L$. It was verified using a representative $45$-cap $5$-flat that if $y$ and $z$ are co-ordinates of $F$ corresponding to the $\{45,45,0\}$ hyperplane direction and any $\{27,27,36\}$ hyperplane direction respectively of $(S_{a},F)$ such that $y$ and $z$ take the value $0$ at the origin of $F$, then the following statements hold: the point count of $(S_{a},F)$ for $(y,z)$ is
\[\left(\begin{array}{ccc}
9 & 0 & 18\\
18 & 0 & 18\\
18 & 0 & 9
\end{array}\right) \textrm{ or } \left(\begin{array}{ccc}
18 & 0 & 9\\
18 & 0 & 18\\
9 & 0 & 18
\end{array}\right),\]
each $18$-point $3$-flat of the form $(y,z) = (b,c)$ is the complement of a union of three disjoint lines, and each $9$-point $3$-flat of the form $(y,z) = (b,c)$ is a $9$-cap $3$-flat.

Suppose that $S_{-1}$ and $S_{1}$ are disjoint; we derive a contradiction.

If the $45$-cap $5$-flats $\widetilde{x}_{1} = \pm 1$ have parallel axis line directions, then $(S_{-1},F)$ and $(S_{1},F)$ have the same hyperplane direction $D$ as a $\{45,45,0\}$ hyperplane direction, so in each $4$-flat of $D$ that does not contain $0 \in F$, there are $45$ points in each of $S_{-1}$ and $S_{1}$, which is impossible since the number of points in each $4$-flat is $3^{4} = 81 < 90 = 2(45)$. So the axis line directions of the $45$-cap $5$-flats $\widetilde{x}_{1} = \pm 1$ are not parallel.

Let the co-ordinates $y_{-1}$ and $y_{1}$ of $F$ correspond via duality to the axis line directions of $\widetilde{x}_{1} = -1$ and $\widetilde{x}_{1} = 1$ respectively, and suppose that $y_{-1}$ and $y_{1}$ take the value $0$ at the origin of $F$. For each $a \in \{\pm 1\}$, in each hyperplane $y_{a} = \pm 1$: there are $45$ points in $S_{a}$ and at least $27$ points in $S_{-a}$, so there are at most $9$ points in neither $S_{a}$ nor $S_{-a}$ (because $S_{-1}$ and $S_{1}$ are disjoint). Therefore, among the $3^{5} - 2(90) = 63$ points in $F - S_{-1} - S_{1}$, at least $63 - 4(9) = 27$ points are in the $3$-flat $(y_{-1},y_{1}) = (0,0)$, which has exactly $3^{3} = 27$ points (as every $3$-flat does); it follows that all the inequalities in this paragraph are equalities.

Therefore, the $y_{-1}$-hyperplane direction of $(S_{1},F)$ and the $y_{1}$-hyperplane direction of $(S_{-1},F)$ have point count $\{27,27,36\}$, so the point counts of $(S_{-1},F)$ and $(S_{1},F)$ for $(y_{-1},y_{1})$ are
\[\left(\begin{array}{ccc}
a & 0 & 27 - a\\
18 & 0 & 18\\
27 - a & 0 & a
\end{array}\right) \textrm{ and }
\left(\begin{array}{ccc}
b & 18 & 27 - b\\
0 & 0 & 0\\
27 - b & 18 & b
\end{array}\right)\]
respectively, for some $a \in \{9,18\}$ and some $b \in \{9,18\}$. Since $S_{-1}$ and $S_{1}$ are disjoint and each $3$-flat has $27$ points in total, we have $a + b \leq 27$ and $(27 - a) + (27 - b) \leq 27$, so $27 \leq a + b \leq 27$, so $a + b = 27$, so $\{a,b\} = \{9,18\}$.

Now the $3$-flat $(y_{-1},y_{1}) = (-1,1)$ in $(S_{-1},F)$ and the $3$-flat $(y_{-1},y_{1}) = (-1,1)$ in $(S_{1},F)$ are, in some order, a $9$-cap $3$-flat and the complement of a union of three disjoint lines, so their $9$- and $18$-point sets are not disjoint, so $S_{-1}$ and $S_{1}$ are not disjoint. We obtain a contradiction.
\end{proof}

\begin{lemma}\label{no18071805173gon}
(a) Suppose that the $4$-flat point count of $C$ for some independent pair $(x_{1},x_{2})$ of co-ordinates is
\[\left(\begin{array}{ccc}
18 & 7 & 18\\
7 & a & *\\
18 & * & *
\end{array}\right),\]
and that each of the $18$-cap $4$-flats $(x_{1},x_{2}) = (-1,-1)$, $(x_{1},x_{2}) = (-1,1)$, and $(x_{1},x_{2}) = (1,1)$ is an $882A_{2}$. It follows that $a \leq 4$. Moreover, if $a \geq 1$, then the relative positions of those three $18$-cap $4$-flats are determined up to isomorphism.

(b) Suppose that the $4$-flat point count of $C$ for some independent pair $(x_{1},x_{2})$ of co-ordinates is
\[\left(\begin{array}{ccc}
18 & 7 & 18\\
7 & a & *\\
17 & * & *
\end{array}\right),\]
and that each of the $18$-cap $4$-flats $(x_{1},x_{2}) = (\pm 1,1)$ is an $882A_{2}$. It follows that $a \leq 4$.
\end{lemma}
\begin{proof}
(a) The $43$-cap $5$-flat $x_{1} = -1$ is a $45$-cap $5$-flat minus two cap points that are in the $4$-flat $(x_{1},x_{2}) = (-1,0)$, and the $43$-cap $5$-flat $x_{2} = 1$ is a $45$-cap $5$-flat minus two cap points  that are in the $4$-flat $(x_{1},x_{2}) = (0,1)$. From Lemma \ref{n5882855para}, it follows that the $18$-cap $4$-flats $(x_{1},x_{2}) = \pm(1,1)$ are translations or point reflections of each other. The former case yields $a = 0$, and the latter case yields $a \leq 4$ (in each case, avoid a line of three cap points).

(b) There is a unique noncap point $P$ in the $4$-flat $(x_{1},x_{2}) = (-1,-1)$ such that $P$ and the cap points in the $5$-flat $x_{1} = -1$ together form a cap in the $5$-flat $x_{1} = -1$. In that new cap, the $18$-cap $4$-flat $(x_{1},x_{2}) = (-1,-1)$ is an $882A_{2}$ by Thackeray \cite[Lemma 3.1]{T22}. Now argue as in part (a).
\end{proof}

\begin{proposition}[Refining \{43,*,43\}]\label{n6res3rd}
Choose an independent pair $(x_{1},x_{2})$ of co-ordinates for $C$. Suppose that the point count of $C$ for $(x_{1},x_{2})$ is among the following, where $a$, $b$, and $c$ are nonnegative integers, and each of the four $4$-flats $(x_{1},x_{2}) = (\pm 1,\pm 1)$ can be obtained from an $882A_{2}$ by removing at most two cap points.
\[\textrm{(a) } \left(\begin{array}{ccc}
18 & a & 18\\
7 & b & 7\\
18 & c & 18
\end{array}\right)
\textrm{ (b) } \left(\begin{array}{ccc}
18 & a & 18\\
7 & b & 8\\
18 & c & 17
\end{array}\right)
\textrm{ (c) } \left(\begin{array}{ccc}
18 & a & 18\\
8 & b & 8\\
17 & c & 17
\end{array}\right)\]
\[\textrm{(d) } \left(\begin{array}{ccc}
18 & a & 17\\
8 & b & 8\\
17 & c & 18
\end{array}\right)
\textrm{ (e) } \left(\begin{array}{ccc}
18 & a & 18\\
7 & b & 9\\
18 & c & 16
\end{array}\right)
\textrm{ (f) } \left(\begin{array}{ccc}
18 & a & 17\\
7 & b & 9\\
18 & c & 17
\end{array}\right)\]
\[\textrm{ (g) } \left(\begin{array}{ccc}
18 & a & 18\\
8 & b & 9\\
17 & c & 16
\end{array}\right)
\textrm{ (h) } \left(\begin{array}{ccc}
18 & a & 17\\
8 & b & 9\\
17 & c & 17
\end{array}\right)
\textrm{ (i) } \left(\begin{array}{ccc}
18 & a & 16\\
8 & b & 9\\
17 & c & 18
\end{array}\right)\]
In each case, $a + b + c \leq 22$ holds. The following restrictions hold for specific matrices.

(a) If $a + b + c \geq 19$, then $b \leq 4$.

(b) If $a + b + c \geq 20$, then $b \leq 4$.

(c) If $a + b + c = 22$, then $(a,b,c) = (6,7,9)$ or $(a,b,c) = (9,4,9)$.

(d) If $a + b + c \geq 21$, then $b \leq 4$.

(e) If $a + b + c = 22$, then $(a,b,c) = (4,9,9)$ or $(a,b,c) = (9,4,9)$.

(f) If $a + b + c = 22$, then $(a,b,c) = (9,4,9)$.

(g) If $a + b + c = 22$, then $(a,b,c)$ is one of $(4,9,9)$, $(6,7,9)$, and $(9,4,9)$.

(h) If $a + b + c = 22$, then $(a,b,c)$ is one of $(4,9,9)$, $(7,7,8)$, and $(9,4,9)$.

(i) If $a + b + c = 22$, then $(a,b,c)$ is one of $(9,9,4)$, $(9,6,7)$, and $(9,4,9)$.
\end{proposition}
\begin{proof}
For all cases, each of $a$, $b$, and $c$ is at most $9$ (avoid a $5$-dimensional cap of size at least $43$ with an impossible hyperplane point count).

We consider matrix (c); argue similarly for the other matrices. Suppose $b \geq 8$ and $c \geq 9$. By Lemma \ref{n4s18882A2repl3pts}, the cap points in the $5$-flats $x_{1} = -1$, $x_{2} = -1$, and $x_{1} = -x_{2}$, together with some point in $(x_{1},x_{2}) = (-1,-1)$ and some point in $(x_{1},x_{2}) = (1,-1)$, form a cap with point count
\[\left(\begin{array}{ccc}
18 & 0 & 0\\
8 & b & 0\\
18 & 9 & 18
\end{array}\right),\]
which contradicts Lemma \ref{no18071805173gon}.

Therefore, we have the following. If $b \geq 8$, then $c \leq 8$ and (by a similar argument) $a \leq 4$, so $a + b + c \leq 4 + 9 + 8 = 21$. If $5 \leq b \leq 7$, then $a \leq 6$ by Lemma \ref{no18071805173gon}, so $a + b + c \leq 6 + 7 + 9 = 22$. If $b \leq 4$, then $a + b + c \leq 9 + 4 + 9 = 22$.
\end{proof}

\begin{proposition}[Refining 5-flat directions]\label{n645Star45}
We have the following.
\begin{itemize}
\item[(a)] In $C$, suppose that the two $5$-flats $\widetilde{x}_{1} = \pm 1$ have exactly $45$ cap points each. It follows that if the two $45$-cap $5$-flats $\widetilde{x}_{1} = \pm 1$ are not point reflections of each other, then the number of cap points in the $5$-flat $\widetilde{x}_{1} = 0$ is at most $6$. That figure of $6$ cannot be lowered.
\item[(b)] In $C$, suppose that the two $5$-flats $\widetilde{x}_{1} = \pm 1$ combined have at least $88$ cap points in total. It follows that if the two $45$-cap $5$-flats obtained by completing the $5$-flats $\widetilde{x}_{1} = \pm 1$ are not point reflections of each other, then the number of cap points in the $5$-flat $\widetilde{x}_{1} = 0$ is at most $14$.
\end{itemize}
\end{proposition}
\begin{proof}
(a) In $C$, a co-ordinate $x_{1}$ can be chosen as in Lemma \ref{n6180918882A2}. Therefore, without loss of generality, there are co-ordinates $x_{2}$ to $x_{5}$ of $C$, with $(\widetilde{x}_{1},x_{1},\ldots,x_{5})$ independent, such that (i) the $5$-flat $\widetilde{x}_{1} = -1$ is Figure 3 of Thackeray \cite{T22} in $(x_{1},\ldots,x_{5})$ co-ordinates and (ii) the $5$-flat $\widetilde{x}_{1} = 1$ is the image under an invertible linear map $T$ of the $5$-flat $\widetilde{x}_{1} = -1$ in $(x_{1},\ldots,x_{5})$ co-ordinates such that $T$ sends each $4$-flat $(\widetilde{x}_{1},x_{1}) = (-1,a)$ to the $4$-flat $(\widetilde{x}_{1},x_{1}) = (1,a)$.

A computer search examined each $T$ in turn. For each $T$, the program found
\begin{itemize}
\item[(i)] the number $n_{0}$ of points $P$ in the $5$-flat $\widetilde{x}_{1} = 0$ such that $P$ is not a midpoint of a line segment with one endpoint in the $5$-flat $\widetilde{x}_{1} = -1$ and the other endpoint in the $5$-flat $\widetilde{x}_{1} = 1$, and
\item[(ii)] the number $n_{2}$ of points $Q$ in the $5$-flat $\widetilde{x}_{1} = 0$ such that $Q$ is a midpoint of at most two such line segments.
\end{itemize}
For each $T$, the program checked whether $n_{0} \geq 6$ holds and $n_{2} \geq 14$ holds; if at least one of those was found to hold, then the program displayed the data of $T$ as well as $n_{0}$ and $n_{2}$. It was thus determined that for each $T$, exactly one of the following is true: (i) $(n_{0},n_{2}) = (0,45)$ (there are $8$ maps $T$ in this case), (ii) $(n_{0},n_{2}) = (22,22)$ (there are $8$ maps $T$ in this case), (iii) $(n_{0},n_{2}) = (6,6)$ (there are $32$ maps $T$ in this case), (iv) $(n_{0},n_{2}) = (2,14)$ (there are $176$ maps $T$ in this case), or (v) both $n_{0} \leq 5$ and $n_{2} \leq 13$ hold.

For each of the $8$ maps $T$ in case (i), the $45$-cap $5$-flat $\widetilde{x}_{1} = 1$ is the image of the $45$-cap $5$-flat $\widetilde{x}_{1} = -1$ under some translation map $U$ from the $6$-flat of $C$ to itself. Since each of the $45$-cap $5$-flats $\widetilde{x}_{1} = \pm 1$ is complete (that is, it cannot be enlarged to form a new cap of the same dimension and greater size by adding cap points), it follows that the $45$ points $Q$ form the image of the $45$-cap $5$-flat $\widetilde{x}_{1} = 1$ under $U$, each of the $45$ points $Q$ is the midpoint of exactly one cap line segment with one endpoint in each of the two $5$-flats $\widetilde{x}_{1} = \pm 1$ (namely, the line segment $U(Q)U^{2}(Q)$), and those $45$ cap line segments are pairwise disjoint.

For each of the $8$ maps $T$ in case (ii), the two $45$-cap $5$-flats $\widetilde{x}_{1} = \pm 1$ are point reflections of each other, the $22$ points $P$ form a cap, and if we combine that cap with the $45$-cap $5$-flats $\widetilde{x}_{1} = \pm 1$, then we obtain a $112$-cap $6$-flat.

An example in case (iii) is in Figure \ref{fign6s96}. For each of the $32$ maps $T$ in case (iii), the six points $P = Q$ form a cap and we obtain a complete $96$-cap $6$-flat. Using the computer search results, it was verified that those $96$-cap $6$-flats are all pairwise isomorphic via isomorphisms that preserve $\widetilde{x}_{1}$ and preserve or negate $x_{1}$; therefore, given any $\{45,45,6\}$ hyperplane direction $D$ of any complete $96$-cap $6$-flat, there is an isomorphism from that cap to Figure \ref{fign6s96} that sends the hyperplanes in $D$ to the $5$-flats $\widetilde{x}_{1} = -1$, $\widetilde{x}_{1} = 0$, and $\widetilde{x}_{1} = 1$ respectively.

\begin{figure}[tbph]
\centering
\includegraphics{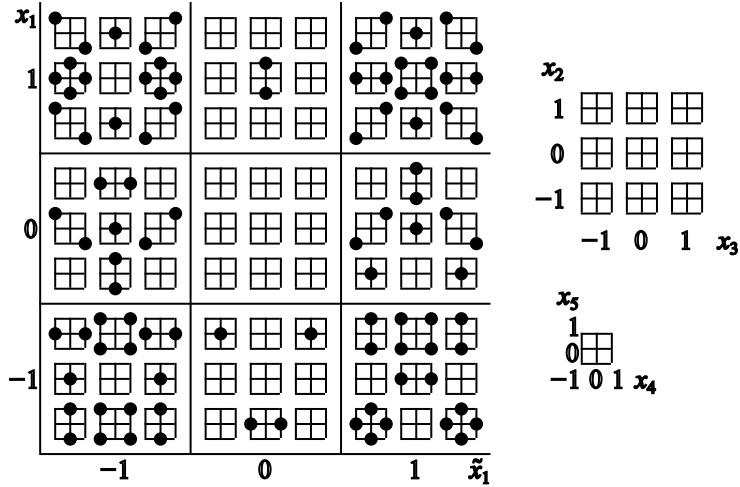}
\caption{A $96$-cap $6$-flat.\label{fign6s96}}
\end{figure}

Part (a) of the Proposition follows, using the information about points $P$ above.

(b) This follows from Theorem \ref{n5s42to45} and the information about points $Q$ in the proof of part (a).
\end{proof}

\begin{proposition}[Refining 5-flat directions]
We have the following.
\begin{itemize}
\item[(a)] In $C$, suppose that the $5$-flats $\widetilde{x}_{1} = -1$ and $\widetilde{x}_{1} = 1$ have exactly $45$ and $40$ cap points respectively, the $40$-cap $5$-flat $\widetilde{x}_{1} = 1$ is the $40$-cap $5$-flat in a $\{40,36,36\}$ hyperplane direction of some $112$-cap $6$-flat, and that $40$-cap $5$-flat has an $\{18,18,4\}$ hyperplane direction in which the $18$-cap $4$-flats are $882A_{2}$ caps. It follows that the number of cap points in the $5$-flat $\widetilde{x}_{1} = 0$ is at most $3$. That figure of $3$ cannot be lowered.
\item[(b)] In $C$, suppose that the union of the two $5$-flats $\widetilde{x}_{1} = \pm 1$ can be obtained from the union of those two $5$-flats in part (a) by removing at most two cap points. It follows that the number of cap points in the $5$-flat $\widetilde{x}_{1} = 0$ is at most $13$.
\end{itemize}
\end{proposition}
\begin{proof}
It was verified that if $\widetilde{C}$ is the $40$-cap $5$-flat $\widetilde{x}_{1} = 1$ as specified in part (a), then the number of $\{18,18,4\}$ hyperplane directions of $\widetilde{C}$ is $10$ and the symmetries of $\widetilde{C}$ act transitively on the set of $\{18,18,4\}$ hyperplane directions of $\widetilde{C}$.

Computer searches were performed as in the proof of Proposition \ref{n645Star45}, where the point count of $C$ for $(\widetilde{x}_{1},x_{1})$ is
\[\left(\begin{array}{ccc}
18 & * & 18\\
9 & * & 4\\
18 & * & 18
\end{array}\right)
\textrm{ or }
\left(\begin{array}{ccc}
15 & * & 18\\
15 & * & 4\\
15 & * & 18
\end{array}\right).\]
A representative $\{18,4,18\}$ direction of the $40$-cap $5$-flat $\widetilde{x}_{1} = 1$ was used as the $x_{1}$-hyperplane direction of that $5$-dimensional subcap, and each of Lemma \ref{n5s45hypclassify}'s four options (i) to (iv) for the $x_{1}$-hyperplane direction of the $45$-cap $5$-flat $\widetilde{x}_{1} = -1$ was considered in turn by the computer search.

For each $T$, the number of points $P$ was at most $3$, and the number of points $Q$ was at most $13$. The search found examples such that there are exactly three points $P$ and those points are not collinear.
\end{proof}

\begin{proposition}[Refining \{45,*,42\}, part 1]\label{n645Star42Tri686}
In $C$, suppose that the numbers of cap points in the two $5$-flats $\widetilde{x}_{1} = -1$ and $\widetilde{x}_{1} = 1$ are $45$ and $42$ respectively, and that the $42$-cap $5$-flat $\widetilde{x}_{1} = 1$ is a $\triangle{}686$. It follows that the number of cap points in the $5$-flat $\widetilde{x}_{1} = 0$ is at most $18$.
\end{proposition}
\begin{proof}
The $42$-cap $5$-flat $\widetilde{x}_{1} = 1$ has a $\{20,16,6\}$ hyperplane direction, so there is a co-ordinate $x_{1}$ of $C$ with $(\widetilde{x}_{1},x_{1})$ independent such that the point count of $C$ for $(\widetilde{x}_{1},x_{1})$ is
\[\left(\begin{array}{ccc}
18 & a & 20\\
9 & b & 6\\
18 & c & 16
\end{array}\right) \textrm{ or } \left(\begin{array}{ccc}
15 & a & 20\\
15 & b & 6\\
15 & c & 16
\end{array}\right).\]
The computer search of Thackeray \cite[Table 1]{T22}, the classification of $44$-cap $5$-flats in Thackeray \cite{T22}, and the classification of $42$-cap $5$-flats above imply that
\begin{itemize}
\item[(i)] The first matrix yields $a \leq 2$, $b \leq 2$, and $c \leq 9$, so $a + b + c \leq 13$; and
\item[(ii)] The second matrix yields $a \leq 6$, $b \leq 6$, and $c \leq 6$, so $a + b + c \leq 18$. \qedhere
\end{itemize}
\end{proof}

\begin{proposition}[Refining \{45,*,42\}, part 2]\label{n645Star4245Min3}
In $C$, suppose that the numbers of cap points in the $5$-flats $\widetilde{x}_{1} = -1$, $\widetilde{x}_{1} = 1$, and $\widetilde{x}_{1} = 0$ are respectively $45$, $42$, and at least $20$. Suppose that the $42$-cap $5$-flat $\widetilde{x}_{1} = 1$ is a $45$-cap $5$-flat minus three cap points. It follows that (i) the $5$-flat $\widetilde{x}_{1} = 0$ has at most $22$ cap points, (ii) the $42$-cap $5$-flat $\widetilde{x}_{1} = 1$ is a point reflection of the $45$-cap $5$-flat $\widetilde{x}_{1} = -1$ minus three cap points, and (iii) $C$ is a subcap of a $112$-cap $6$-flat.
\end{proposition}
\begin{proof}
By Lemma \ref{n6180918882A2}, there is a co-ordinate $x_{1}$ of $C$ with $(\widetilde{x}_{1},x_{1})$ independent such that the union of the two $5$-dimensional caps $\widetilde{x}_{1} = \pm 1$ in $C$ is obtained by removing at most three cap points from a $90$-cap $6$-flat $\widetilde{C}$, with the same underlying $6$-flat as $C$, such that the point count of $\widetilde{C}$ for $(\widetilde{x}_{1},x_{1})$ is
\[\left(\begin{array}{ccc}
18 & 0 & 18\\
9 & 0 & 9\\
18 & 0 & 18
\end{array}\right)\]
and the four $18$-cap $4$-flats $(\widetilde{x}_{1},x_{1}) = (\pm 1,\pm 1)$ in $\widetilde{C}$ are all $882A_{2}$ caps. From the $(18,9,18)$ columns of that point count matrix, it follows that for each $a \in \{\pm 1\}$, the $882$ cube direction of each of the two $882A_{2}$ caps $(\widetilde{x}_{1},x_{1}) = (a,\pm 1)$ in $\widetilde{C}$ is parallel to the $855$ cube direction of the other $882A_{2}$.

It follows that without loss of generality (possibly negating $x_{1}$), the point count of $C$ for $(\widetilde{x}_{1},x_{1})$ is among
\[\begin{array}{c}
\left(\begin{array}{ccc}
18 & a & 18\\
9 & b & 6\\
18 & c & 18
\end{array}\right), \left(\begin{array}{ccc}
18 & a & 18\\
9 & b & 7\\
18 & c & 17
\end{array}\right), \left(\begin{array}{ccc}
18 & a & 18\\
9 & b & 8\\
18 & c & 16
\end{array}\right),\\
\left(\begin{array}{ccc}
18 & a & 18\\
9 & b & 9\\
18 & c & 15
\end{array}\right), \left(\begin{array}{ccc}
18 & a & 17\\
9 & b & 8\\
18 & c & 17
\end{array}\right), \textrm{ and } \left(\begin{array}{ccc}
18 & a & 17\\
9 & b & 9\\
18 & c & 16
\end{array}\right).
\end{array}\]
In each case, we have $a \leq 9$ and $c \leq 9$ (no $5$-dimensional cap of size at least $43$ has a hyperplane direction with point count $\{18,*,10\}$).

In $C$, if each of the $5$-flats $x_{1} = 1$, $\widetilde{x}_{1} - x_{1} = 0$, $\widetilde{x}_{1} + x_{1} = 0$, and $x_{1} = -1$ has at most $41$ cap points, then each of the six point count matrices yields $a + b + c \leq 19$ and a contradiction. (For example, the last of the six matrices yields $a \leq 41 - (18 + 17) = 6$, $b \leq 41 - (18 + 17) = 6$, and $c \leq 41 - (18 + 16) = 7$, so $a + b + c \leq 6 + 6 + 7 = 19$.) Therefore, at least one of the four caps in the $5$-flats $x_{1} = 1$, $\widetilde{x}_{1} - x_{1} = 0$, $\widetilde{x}_{1} + x_{1} = 0$, and $x_{1} = -1$ of $C$ has at least $42$ cap points.

Let $C_{5}$ be such a $5$-dimensional cap. From the point count matrices above, $C_{5}$ has an $\{18,18,*\}$, $\{18,17,*\}$, $\{18,16,*\}$, or $\{18,15,*\}$ hyperplane direction in which the $18$-, $17$-, $16$-, and $15$-cap $4$-flats are $882A_{2}$ caps minus at most three cap points each. It follows that $C_{5}$ is a $45$-cap $5$-flat minus at most three cap points: in each $\{18,18,6\}$ or $\{18,17,7\}$ hyperplane direction of a $\triangle{}686$, each $18$-cap $4$-flat is a $963B$ or a $981C$ respectively, not an $882A_{2}$, and a $\triangle{}686$ has neither $\{18,16,8\}$ nor $\{18,15,9\}$ hyperplane directions.

By Lemma \ref{n4s18882A2repl3pts}, there is a unique way to add cap points to the $4$-flats $(\widetilde{x}_{1},x_{1}) = (1,\pm 1)$ of $C$ to obtain $882A_{2}$ caps, so those $882A_{2}$ caps must be the caps in the same positions in $\widetilde{C}$. Therefore, by Lemma \ref{n5882855para}, the nine-$2$s $2$-flat directions of the $882A_{2}$ caps $(\widetilde{x}_{1},x_{1}) = (\pm 1,\pm 1)$ in $\widetilde{C}$ are parallel to one another, and the $882$ and $855$ cube directions of each of those four $882A_{2}$ caps are parallel to those directions, in some order, of each other such $882A_{2}$ cap.

Suppose that the underlying $5$-flat of $C_{5}$ is $\widetilde{x}_{1} - x_{1} = 0$ or $\widetilde{x}_{1} + x_{1} = 0$. We have $b \leq 9$, and Lemma \ref{n5882855para} implies the following: the two $882A_{2}$ caps $(\widetilde{x}_{1},x_{1}) = (\pm 1,-1)$ of $\widetilde{C}$ are translations or point reflections of each other, and the two $882A_{2}$ caps $(\widetilde{x}_{1},x_{1}) = (\pm 1,1)$ of $\widetilde{C}$ are translations or point reflections of each other. Lemma \ref{n5Double882A2Min3} yields $a \leq 4$ and $c \leq 5$, so $a + b + c \leq 4 + 9 + 5 = 18$, so we have a contradiction.

Therefore, the underlying $5$-flat of $C_{5}$ is $x_{1} = 1$ or $x_{1} = -1$. Lemma \ref{n5882855para} implies that the $882A_{2}$ caps $(\widetilde{x}_{1},x_{1}) = \pm(1,-1)$ in $\widetilde{C}$ are translations or point reflections of each other, and the $882A_{2}$ caps $(\widetilde{x}_{1},x_{1}) = \pm(1,1)$ in $\widetilde{C}$ are translations or point reflections of each other. We have $b \leq 4$ by Lemma \ref{n5Double882A2Min3}, so $a + b + c \leq 9 + 4 + 9 = 22$, which yields (i).

We have $a + c \geq 20 - 4 = 16$, so $(a,c)$ is among $(7,9)$, $(8,8)$, $(9,7)$, $(8,9)$, $(9,8)$, and $(9,9)$. If $a + c = 16$, then $b \geq 20 - 16 = 4$; if $\{a,c\} = \{8,9\}$, then $b \geq 20 - 17 = 3$; and if $a = c = 9$, then $b \geq 20 - 18 = 2$.

By Lemma \ref{n5Double882A2Min3}, the $882A_{2}$ caps $(\widetilde{x}_{1},x_{1}) = \pm(1,1)$ in $\widetilde{C}$ are point reflections of each other (in each of the six point count matrices, the bottom-left-to-top-right diagonal is $(18,b,18)$ or $(18,b,17)$, and $b \geq 2$). Therefore, the square of midpoints of $(\widetilde{x}_{1},x_{1}) = (1,1)$ in $\widetilde{C}$ is the image of the square of midpoints of $(\widetilde{x}_{1},x_{1}) = (-1,-1)$ in $\widetilde{C}$ under some translation map $U$, and all cap points in the $4$-flat $(\widetilde{x}_{1},x_{1}) = (0,0)$ in $C$ are in the image $M$ under $U$ of the square of midpoints of the $882A_{2}$ cap $(\widetilde{x}_{1},x_{1}) = (1,1)$ in $\widetilde{C}$. The $2$-flat $F$ that contains $M$ is a translation of a $2$-flat in the nine-$2$s $2$-flat directions of the $882A_{2}$ caps $(\widetilde{x}_{1},x_{1}) = (\pm 1,\pm 1)$ in $\widetilde{C}$.

Suppose $b \geq 3$. The $882A_{2}$ caps $(\widetilde{x}_{1},x_{1}) = \pm(1,-1)$ in $\widetilde{C}$ must be point reflections of each other (if they were translations of each other, then at most two of the cap points of $C$ in $(\widetilde{x}_{1},x_{1}) = (0,0)$ could be in $F$). Without loss of generality, let the $5$-flat $\widetilde{x}_{1} = -1$ of $C$ be as in Figure \ref{fign6s112rep}. The $18$-cap $4$-flat $(x_{1},x_{2}) = (1,1)$ of $\widetilde{C}$ is as in that figure without loss of generality, and the $18$-cap $4$-flat $(x_{1},x_{2}) = (1,\pm 1)$ of $\widetilde{C}$ is a translation of the $18$-cap $4$-flat $(x_{1},x_{2}) = (1,\pm 1)$ in the figure by some vector $v$. The square $M$ consists of the cap points in the $4$-flat $(x_{1},x_{2}) = (0,0)$ in the figure. All cap points in the $4$-flat $(x_{1},x_{2}) = (0,0)$ of $C$ are among those of $M$, and among those in the image of $M$ under translation by $-v$. If two squares that are translations of each other intersect in at least three points, then they coincide, so $v = 0$. That implies parts (ii) and (iii).

We may therefore assume $b = 2$; it follows that $a = c = 9$.

If the $882A_{2}$ caps $(\widetilde{x}_{1},x_{1}) = \pm(1,-1)$ in $\widetilde{C}$ are translations of each other, then the $9$-cap $4$-flats of $C$ in $(\widetilde{x}_{1},x_{1}) = (-1,0)$ and $(\widetilde{x}_{1},x_{1}) = (0,-1)$ determine the positions of the standard squares of the $882A_{2}$ cap of $\widetilde{C}$ in $(\widetilde{x}_{1},x_{1}) = (1,1)$, and in all cases we obtain at least one line of three cap points and a contradiction.

Therefore, the $882A_{2}$ caps $(\widetilde{x}_{1},x_{1}) = \pm(1,-1)$ in $\widetilde{C}$ are point reflections of each other, and arguing as before using the $4$-flats $(\widetilde{x}_{1},x_{1}) = (-1,0)$ and $(\widetilde{x}_{1},x_{1}) = (0,-1)$, we obtain parts (ii) and (iii).
\end{proof}

\begin{theorem}\label{n6s110}
Every $110$-cap $6$-flat is a $112$-cap $6$-flat minus two cap points.
\end{theorem}
\begin{proof}
Let $C$ be a $110$-cap $6$-flat. From the standard diagram in Figure \ref{fign6stddiag}, every $110$-cap $6$-flat has at least one $\{45,45,20\}$, $\{45,44,21\}$, $\{45,43,22\}$, $\{44,44,22\}$, $\{45,42,23\}$, $\{44,43,23\}$, or $\{43,43,24\}$ hyperplane direction. Choose $\widetilde{x}_{1}$ to be a co-ordinate of such a hyperplane direction $D$ so that the numbers of cap points in the $5$-flats $\widetilde{x}_{1} = -1$, $\widetilde{x}_{1} = 1$, and $\widetilde{x}_{1} = 0$ are in nonincreasing order.

\begin{figure}
\includegraphics{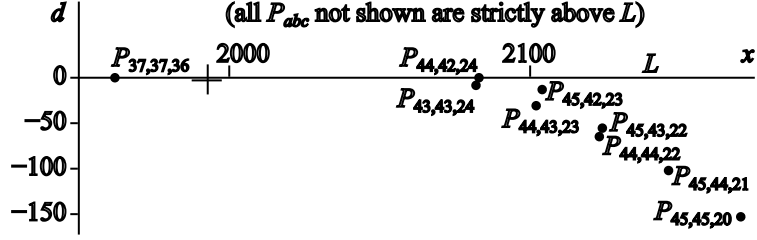}
\caption{Standard diagram for $(n,s) = (6,110)$. The line $L$ has equation $121y = 4068x - 5237136$ and passes through $P_{44,42,24}$ and $P_{37,37,36}$.\label{fign6stddiag}}
\end{figure}

If $D$ has point count $\{45,45,20\}$, $\{45,44,21\}$, $\{45,43,22\}$, or $\{44,44,22\}$, then by Proposition \ref{n645Star45}, the $5$-flats $\widetilde{x}_{1} = \pm 1$ are subcaps of $45$-cap $5$-flats that are point reflections of each other, and we can argue as before to show that (i) for some co-ordinate $x_{1}$ of $C$, the point count of $C$ for $(\widetilde{x}_{1},x_{1})$ can be obtained from
\[\left(\begin{array}{ccc}
18 & 9 & 18\\
9 & 4 & 9\\
18 & 9 & 18
\end{array}\right)\]
by subtracting $2$ in total from the nine entries combined, and (ii) $C$ is a $112$-cap $6$-flat minus two cap points.

If $D$ has point count $\{45,42,23\}$, then Propositions \ref{n645Star42Tri686} and \ref{n645Star4245Min3} yield a contradiction.

If $D$ has point count $\{44,43,23\}$, then by Lemma \ref{n6180918882A2}, there is a co-ordinate $x_{1}$ of $C$ such that the point count of $C$ for $(\widetilde{x}_{1},x_{1})$ (possibly after negating $\widetilde{x}_{1}$) can be obtained from some point count matrix in Proposition \ref{n6res3rd} by adding $1$ to some entry in the left or right column, so by that theorem, the $5$-flat $\widetilde{x}_{1} = 0$ of $C$ has at most $22$ cap points, and we obtain a contradiction.

Suppose that $D$ has point count $\{43,43,24\}$. Arguing as before, for some co-ordinate $x_{1}$ of $C$, the point count of $C$ for $(\widetilde{x}_{1},x_{1})$ (possibly after negating $\widetilde{x}_{1}$) is, without loss of generality, either among the point count matrices in Proposition \ref{n6res3rd} -- in which case, by that theorem, the $5$-flat $\widetilde{x}_{1} = 0$ of $C$ has at most $22$ cap points, and we have a contradiction -- or among the following four point count matrices where each of the $4$-dimensional caps $(\widetilde{x}_{1},x_{1}) = (\pm 1,\pm 1)$ can be completed to an $882A_{2}$:
\[\left(\begin{array}{ccc}
17 & a & 17\\
9 & b & 9\\
17 & c & 17
\end{array}\right), \left(\begin{array}{ccc}
17 & a & 18\\
9 & b & 9\\
17 & c & 16
\end{array}\right), \left(\begin{array}{ccc}
16 & a & 18\\
9 & b & 9\\
18 & c & 16
\end{array}\right), \left(\begin{array}{ccc}
18 & a & 18\\
9 & b & 9\\
16 & c & 16
\end{array}\right).\]
In each case, we have $b \leq 9$, and if $b \geq 6$ then the completed $882A_{2}$ caps in $(\widetilde{x}_{1},x_{1}) = (\pm 1,1)$ are translations or point reflections of each other, and similarly for $(\widetilde{x}_{1},x_{1}) = (\pm 1,-1)$, so $a \leq 5$ and $c \leq 5$, so $a + b + c \leq 5 + 9 + 5 = 19$, which gives a contradiction. Therefore, in each case, we have $b \leq 5$, $a \leq 9$, and $c \leq 9$ (for the last of the four point count matrices, note that no $42$-cap $5$-flat has a $\{16,16,10\}$ hyperplane direction), so $a + b + c \leq 23$, which gives a contradiction.
\end{proof}

\section{Dimension 7}

\begin{proposition}[Refining 6-flat directions]\label{n7s222}
Let the $7$-dimensional cap $C$ have co-ordinate $x_{1}$.
\begin{itemize}
\item[(a)] Suppose that the two $6$-flats $x_{1} = \pm 1$ have exactly $112$ cap points each. It follows that the $6$-flat $x_{1} = 0$ has at most $34$ cap points.
\item[(b)] Suppose that the two $6$-flats $x_{1} = \pm 1$ combined have at least $222$ points in total. It follows that the $6$-flat $x_{1} = 0$ has at most $55$ cap points.
\end{itemize}
\end{proposition}
\begin{proof}
(a) For some co-ordinate $x_{2}$ of $C$ with $(x_{1},x_{2})$ independent, the point count of $C$ for $(x_{1},x_{2})$ is one of the following two options.

\textbf{Option 1}: $\displaystyle \left(\begin{array}{ccc}
45 & a & 45\\
22 & b & 22\\
45 & c & 45
\end{array}\right)$

The $45$-cap $5$-flats $(x_{1},x_{2}) = (-1,\pm 1)$ are point reflections of each other, and the $45$-cap $5$-flats $(x_{1},x_{2}) = (1,\pm 1)$ are point reflections of each other.

If neither of $(x_{1},x_{2}) = (1,\pm 1)$ is a point reflection of $(x_{1},x_{2}) = (-1,-1)$, then each of $a$, $b$, and $c$ is at most $6$, so $a + b + c \leq 3(6) = 18$.

If $(x_{1},x_{2}) = (1,-1)$ or $(x_{1},x_{2}) = (1,1)$ is a point reflection of $(x_{1},x_{2}) = (-1,-1)$, then up to isomorphism, the relative positions of the $6$-flats $x_{1} = \pm 1$ are among $2 \times 3^{5}$ options. (Fix the $6$-flat $x_{1} = -1$; the $6$-flat $x_{1} = 1$ is obtained by starting with a point reflection or translation of $x_{1} = -1$ and then, for some $5$-dimensional vector $v$, adding $-v$, $0$, and $v$ respectively to the $5$-flats $(x_{1},x_{2}) = (1,-1)$, $(x_{1},x_{2}) = (1,0)$, and $(x_{1},x_{2}) = (1,1)$ respectively.) A computer search verified that for each of those options,
\begin{itemize}
\item[(i)] every point in the $6$-flat $x_{1} = 0$ is the midpoint of a cap line segment with one endpoint in each of the $6$-flats $x_{1} = \pm 1$, so there are no cap points in the $6$-flat $x_{1} = 0$;
\item[(ii)] if the $112$-cap $6$-flats $x_{1} = \pm 1$ are not translations of each other, then there are at most two points $Q$ in the $6$-flat $x_{1} = 0$ such that $Q$ is the midpoint of at most two cap line segments with one endpoint in each of the $6$-flats $x_{1} = \pm 1$; and
\item[(iii)] if the $112$-cap $6$-flat $x_{1} = 1$ is the image of $x_{1} = -1$ under a translation map $U$, then there are exactly $112$ points $Q$ in $x_{1} = 0$ such that $Q$ is the midpoint of at most two cap line segments with one endpoint in each of $x_{1} = \pm 1$ (it follows that those $112$ points $Q$ are precisely the images of the cap points of $x_{1} = 1$ under $U$, each of the $112$ points $Q$ is the midpoint of exactly one such cap line segment, namely $U(Q)U^{-1}(Q)$, and those $112$ cap line segments $U(Q)U^{-1}(Q)$ are pairwise disjoint).
\end{itemize}

\textbf{Option 2}: $\displaystyle \left(\begin{array}{ccc}
45 & a & 36\\
22 & b & 40\\
45 & c & 36
\end{array}\right)$

Without loss of generality, suppose that $C$ is not included in Option $1$. There are two cases.
\begin{itemize}
\item Case 1: It is impossible for the $40$ in the point count of Option $2$ to correspond to a $40$-cap $5$-flat with an $\{18,18,4\}$ hyperplane direction in which the $18$-cap $4$-flats are $882A_{2}$ caps: In this case, the relative positions of the $112$-cap $6$-flats $x_{1} = \pm 1$ are strongly restricted up to isomorphism (the $56$ hyperplane directions of $x_{1} = -1$ with point count $\{45,22,45\}$ are translations of the $56$ hyperplane directions of $x_{1} = 1$ with point count $\{40,36,36\}$ in which the $40$-cap $5$-flat has no $\{18,18,4\}$ hyperplane direction in which the $18$-cap $4$-flats are $882A_{2}$ caps, and that statement also holds with $x_{1} = -1$ and $x_{1} = 1$ swapped); a computer search verified that in this case, the $6$-flat $x_{1} = 0$ has at most one cap point and at most one point $Q$ such that $Q$ is the midpoint of at most two cap line segments with one endpoint in each of $x_{1} = \pm 1$.
\item Case 2: We may take the $40$ in the point count of Option $2$ to correspond to a $40$-cap $5$-flat with an $\{18,18,4\}$ hyperplane direction in which the $18$-cap $4$-flats are $882A_{2}$ caps: Taking that scenario, we have $a \leq 3$, $c \leq 3$, and $b \leq 109 - (45 + 36) = 28$ (because no $110$-cap $6$-flat has a $\{45,36,29\}$ hyperplane direction), so $a + b + c \leq 3 + 28 + 3 = 34$.
\end{itemize}

(b) The $6$-flats $x_{1} = \pm 1$ can be obtained from those in part (a) by removing at most two cap points in total.

Suppose that the $6$-flats $x_{1} = \pm 1$ together are obtained from those in Option $1$ in part (a) by removing at most two cap points, and the numbers of cap points in the $5$-flats $(x_{1},x_{2}) = (0,1)$, $(x_{1},x_{2}) = (0,0)$, and $(x_{1},x_{2}) = (0,-1)$ are respectively $a$, $b$, and $c$.

If neither of the completions of $(x_{1},x_{2}) = (1,\pm 1)$ to $45$-cap $5$-flats is a point reflection of the completion of $(x_{1},x_{2}) = (-1,-1)$ to a $45$-cap $5$-flat, then each of $a$, $b$, and $c$ is at most $14$, so $a + b + c \leq 3(14) = 42$.

If the completion of $(x_{1},x_{2}) = (1,-1)$ or $(x_{1},x_{2}) = (1,1)$ to a $45$-cap $5$-flat is a point reflection of the completion of $(x_{1},x_{2}) = (-1,-1)$ to a $45$-cap $5$-flat, then it follows from the information about points $Q$ in part (a) that $a + b + c \leq 2$.

Suppose that the $6$-flats $x_{1} = \pm 1$ together are obtained from those in Option $2$ in part (a) by removing at most two cap points, and define $a$, $b$, and $c$ as before. We have $a \leq 13$, $c \leq 13$, and $b \leq 109 - (45 + 36 - 1) = 29$ (in at least one of the two $(45,b,36)$ diagonals, the number of cap points removed is at most $2/2 = 1$; no $110$-cap $6$-flat has a $\{45,36,29\}$, $\{45,35,30\}$, or $\{44,36,30\}$ hyperplane direction), so $a + b + c \leq 13 + 29 + 13 = 55$.
\end{proof}

\begin{theorem}
There are no $289$-cap $7$-flats.
\end{theorem}
\begin{proof}
By the standard diagram in Figure \ref{fign7stddiag}, each $289$-cap $7$-flat has at least one hyperplane direction with point count among $\{112,112,65\}$, $\{112,111,66\}$, $\{112,110,67\}$, and $\{111,111,67\}$. Each of those options is ruled out by Proposition \ref{n7s222} together with Theorem \ref{n6s110}.
\end{proof}
\begin{figure}
\includegraphics{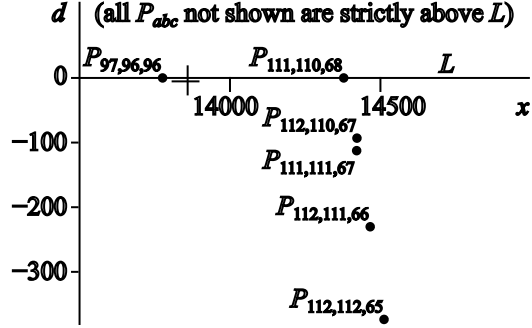}
\caption{Standard diagram for $(n,s) = (7,289)$. The line $L$ has equation $86y = 7793x - 70101168$ and passes through $P_{111,110,68}$ and $P_{97,96,96}$.\label{fign7stddiag}}
\end{figure}

\section*{Acknowledgements}

Many thanks to my postdoctoral supervisor Prof.\@ James Raftery, to Prof.\@ Roumen Anguelov, to Prof.\@ Jan Harm van der Walt, to Prof.\@ Mapundi Banda, to Prof.\@ Anton Str\"{o}h, and to everyone else at the University of Pretoria, for all their generous support at this extraordinary time.

This work was supported by the UP Post-Doctoral Fellowship Programme administered by the University of Pretoria [grant number A0X 816].

Computer searches were carried out using Java programs on the Eclipse IDE software.

Many thanks to my mother Dr Anne Thackeray for letting me use her computer together with my own to perform some computer searches, and to all of my family for everything.


\begin{thebibliography}{9}
\bibitem{DM03} B. L. Davis and D. Maclagan. The card game {S}et. \emph{{M}ath. {I}ntell.}, 25(3):33--40, 2003. $<$https://doi.org/10.1007/BF02984846$>$.
\bibitem{EFLS02} Y. Edel, S. Ferret, I. Landjev, and L. Storme. The classification of the largest caps in {AG}(5,3). \emph{{J}. {C}omb. {T}heory {S}er. {A}}, 99(1):95--110, 2002. $<$https://doi.org/10.1006/jcta.2002.3261$>$.
\bibitem{EG17} J. S. Ellenberg and D. Gijswijt, On large subsets of $\mathbb{F}_{q}^{n}$ with no three-term arithmetic progression, \emph{{A}nn. {M}ath. 2nd {S}er.}, 185(1):339--343, 2017. $<$https://doi.org/10.4007/annals.2017.185.1.8$>$.
\bibitem{P08} A. Potechin, Maximal caps in {AG}(6,3), \emph{{D}es. {C}odes {C}ryptogr.}, 46(3):243--259, 2008. $<$https://doi.org/10.1007/s10623-007-9132-z$>$.
\bibitem{T21} H. {(M.)} R. Thackeray, The cap set problem and standard diagrams, \emph{{D}iscrete {M}ath.}, 344(11):2021. Article 112558, $<$https://doi.org/10.1016/j.disc.2021.112558$>$.
\bibitem{T22} H. {(M.)} R. Thackeray, The cap set problem: 41-cap 5-flats, 2022. Preprint.
\end{thebibliography}
\end{document}